\documentclass[11pt,twoside,dvipsnames]{amsart}
\usepackage[utf8]{inputenc}
\usepackage[english]{babel}
\usepackage[T1]{fontenc}
\usepackage{xcolor}
\usepackage{amssymb}
\usepackage{geometry}
\usepackage[bookmarks=true]{hyperref}
\usepackage{subcaption}

\newtheorem{thm}{Theorem}[section]
\newtheorem{prop}[thm]{Proposition}
\newtheorem{cor}[thm]{Corollary}

\theoremstyle{definition}

\newtheorem{ex}[thm]{Example}

\newtheorem{problem}[thm]{Problem}

\newcommand{\hgt}{{\rm ht}\,}

\newcommand{\ara}{{\rm ara}\,}

\newcommand{\cd}{{\rm cd}\,}
\newcommand{\chara}{{\rm char}\,}

\title[On determinantal ideals and algebraic dependence]{On determinantal ideals and algebraic dependence}

\author{Margherita Barile and Antonio Macchia}
\address{Dipartimento di Matematica, Universit\`a degli Studi di Bari ``Aldo Moro'', Via Orabona 4, 70125 Bari, Italy}
\email[M. Barile]{margherita.barile@uniba.it}
\email[A. Macchia]{macchia.antonello@gmail.com}

\begin{document}

\begin{abstract}
Let $X$ be a matrix with entries in a polynomial ring over an algebraically closed field $K$. We prove that, if the entries of $X$ outside some $(t \times t)$-submatrix are algebraically dependent over $K$,  the arithmetical rank of the ideal $I_t(X)$ of $t$-minors of $X$ drops at least by one with respect to the generic case; under suitable assumptions, it drops at least by $k$  if $X$ has $k$ zero entries.  This upper bound turns out to be sharp if $\mathrm{char}\, K=0$, since it then coincides with the lower bound provided by the local cohomological dimension.
\end{abstract}

\maketitle

\noindent {\bf Mathematics Subject Classification (2010):} 13A15, 14M10, 14M12, 13D45.

\noindent {\bf Keywords:} Determinantal ideals, algebraic dependence, arithmetical rank, sparse matrices.

\section{Introduction}

Let $X$ be an $(m \times n)$-matrix, where $m \leq n$, with entries in the polynomial ring $R$ over a field $K$. A celebrated result by Bruns and Schw\"anzl \cite[Theorem 1]{BS90} states that, if $X$ is generic (i.e., its entries are algebraically independent over $K$), then  for every $t=1,\dots,m$, the ideal $I_t(X)$ generated by the $t$-minors of $X$ can be generated by $mn-t^2+1$, but not fewer, elements up to radical. This means that there exist $q_1(X),\dots,q_{mn-t^2+1}(X) \in R$ such that
\[
\sqrt{I_t(X)}=\sqrt{(q_1(X),\dots, q_{mn-t^2+1}(X))}
\]
and $mn-t^2+1$ is the minimum number of elements for which the above equality holds. This result is independent of the field.

Given an ideal $I$ in a Noetherian ring $S$, the minimum number of elements of $S$ that generate an ideal whose radical is the same as $I$ is called the \textit{arithmetical rank} of $I$ and denoted by $\ara I$. In general, the following inequalities hold (see, e.g., \cite[Proposition 9.2]{ILLMMSW07}):
\[
\hgt I \leq \cd I \leq \ara I,
\]
where $\hgt I$ is the height of $I$, $\cd I = \max \{i \in \mathbb Z : H^i_I(S) \neq 0\}$ is the \textit{cohomological dimension} of $I$ and $H^i_I(S)$ denotes the $i$-th local cohomology module of $S$ with support in $I$.

Lyubeznik, Singh and Walther  asked in \cite[Question 8.1]{LSW16} whether the arithmetical rank of $I_t(X)$  is smaller than in the generic case, when the entries of $X$ are algebraically dependent. In \cite{BCMM14} the second author and others studied the case of $(2 \times n)$-matrices of linearly dependent linear forms and gave a positive answer  for large classes of examples.

In the present paper we consider matrices of arbitrary size whose entries are algebraically dependent over an algebraically closed field. In our main result, Theorem \ref{main}, we prove that if the entries of $X$ lying outside some $(t\times t)$-submatrix are algebraically dependent over $K$, then $\ara I_t(X) \leq mn-t^2$. In particular this holds if some entry of $X$ is zero.

In Section \ref{S.sparseMatrices}, we improve Theorem \ref{main} for \textit{sparse generic matrices}, i.e, matrices whose entries are pairwise distinct variables and zeros. Sparse determinantal ideals have recently been considered in \cite{B11}, where a minimal free resolution is computed for the ideals of maximal minors. In Proposition \ref{P.antidiagonals} we prove that, if $X$ is a sparse matrix with $k$ zeros, where $k \leq \min\{2t+1,m+n-2t\}$ and the zeros are placed on consecutive antidiagonals starting from the upper-left corner, then the arithmetical rank of $I_t(X)$ drops at least by $k$ with respect to the generic case. We actually have that, in characteristic $0$, $\ara I_t(X)=\cd I_t(X)=mn-t^2-k+1$, which, for $k=0$, gives the equality for generic matrices proven in \cite{BS90}. Our result is sharp, because, in general, it fails to be true  if the number $k$ of zeros exceeds the prescribed upper bound.

A similar result holds for the ideal of maximal minors of a sparse matrix with exactly $k \leq n-m$ zeros,  see Proposition \ref{P.kZerosCols}.

In Section \ref{S.2nmatrices}  we consider the case of $(2\times n)$-matrices of linearly dependent linear forms, for which we give a complete positive answer to \cite[Question 8.1]{LSW16}.


\section{Preliminaries}

Let $m$ and $n$ be positive integers, where $m\leq n$. Let $R$ be a commutative unit ring, and $A$ an $(m\times n)$-matrix with entries in $R$. For all positive integers $t$ such that $t\leq m$, we consider the $t$-minors of $A$. More precisely, for all sequences of indices $1\leq a_1<a_2< \cdots < a_t\leq m$ and $1\leq b_1<b_2<\cdots<b_t\leq n$, we denote by $[a_1,\dots, a_t|b_1,\dots, b_t]$ the determinant of the submatrix of $A$ formed by the rows of indices $a_1,\dots, a_t$ and the columns of indices $b_1,\dots, b_t$. By abuse of terminology, the term {\it minor} will also be referred to this submatrix. The following facts are taken from Chapters 4 and 5 of the monograph by Bruns and Vetter \cite{BV90}, to which we refer for further details.  The set of all minors of $A$ can be endowed with the following partial order. We set
\[
[a_1,\dots, a_t|b_1,\dots, b_t]\leq [c_1,\dots, c_u|d_1,\dots, d_u],
\]
if either $t>u$ or $t=u$ and $a_i\leq c_i$, $b_i\leq d_i$ for all $i=1,\dots, t$. In this poset, all maximal chains with a fixed bottom and a fixed top have the same length. We will call {\it rank} of an element the cardinality of all maximal chains having this element as the top.
The maximum  $t$-minor is $[m-t+1, \dots, m|n-t+1,\dots, n]$, the minimum $t$-minor is $[1,\dots, t|1,\dots, t]$, all intermediate poset elements are $t$-minors.
Given a $t$-minor $[a_1,\dots, a_t|b_1,\dots, b_t]>[1,\dots, t|1,\dots, t]$, its lower neighbours are
\begin{list}{}{}
\item{(i)} the $t$-minors obtained by lowering one of its indices by one;
\item{(ii)} if $a_t<m$ and $b_t<n$, the $(t+1)$-minor $[a_1,\dots, a_t, m|b_1,\dots, b_t, n]$.
\end{list}
Hence the distance between two $t$-minors is measured by the difference between the sums of their row and column indices. Since the rank of  $[m-t+1, \dots, m|n-t+1,\dots, n]$ is $mn-t^2+1$, the rank of $[1,\dots, t|1,\dots, t]$ is $mn+t^2-t(m+n)+1$.

For all $h=1,\dots, mn-t^2+1$ let $q_h(A)$ be the sum of all minors of $A$ having rank $h$. We thus have that $q_{mn-t^2+1}(A)=[m-t+1, \dots, m|n-t+1,\dots, n]$. Since the maximum rank of the minors of order greater than $t$ is $mn-t^2-2t= mn-t^2+1 -(2t+1)$,  for all $h=1,\dots, 2t+1$, $q_{mn-t^2-h+2}(A)$ is a sum of $t$-minors.
Let $I_t(A)$ be the ideal of $R$ generated by all $t$-minors of $A$. Then, according to \cite[Corollary 5.21]{BV90},
\begin{equation}\label{1}
\sqrt{I_t(A)}=\sqrt{(q_1(A),\dots, q_{mn-t^2+1}(A))},
\end{equation}
so that $\ara I_t(A)\leq mn-t^2+1$. More precisely, we have that, for all $h=1,\dots, mn-t^2+1$, if $I^{(h)}_t(A)$ is the ideal of $R$ generated by all minors of $A$ whose rank is at most $h$, then
\begin{equation}\label{2}
\sqrt{I^{(h)}_t(A)}=\sqrt{(q_1(A),\dots, q_h(A))}.
\end{equation}
By \cite[Corollary, p. 440]{BS90}, equality holds in (\ref{1}) if $A$ is a generic matrix of indeterminates over a field of characteristic zero, since, in this case, cd\,$I_t(A)=mn-t^2+1$.

\section{The Main Theorem}

Let $X$ be an $(m\times n)$-matrix ($m\leq n$) with entries in the polynomial ring $R=K[x_1,\dots, x_N]$ over the algebraically closed field $K$. For all $h=1,\dots, mn-t^2+1$, let $q_h=q_h(X)$.

\begin{thm}\label{main}
Let $t$ be an integer such that $1\leq t\leq m$ and $t<n$. If the entries of $X$ outside the minor $\Delta=[m-t+1,\dots, m|n-t+1,\dots, n]$ are algebraically dependent over $K$, then
\[
\ara I_t(X) \leq mn-t^2.
\]
\end{thm}

\begin{proof}
Let $u_1,\dots, u_k$ be algebraically dependent entries of $X$ outside $\Delta$, and let $F$ be a nonzero polynomial over $K$ in $k$ indeterminates - which will be denoted by the letter $y$ - such that $F(u_1,\dots, u_k)=0$ ($\ast$). We perform the following algorithm.\newline
{\bf Step 1} If there is some indeterminate $y$ such that $y$ divides some, but not all monomials of $F$, write $F=F'+yF''$, where no monomial of $F'$ is divisible by $y$. Then replace $F$ by $F'$, which by assumption is nonzero, and return to Step 1. Else (i.e., if all monomials of $F$ have the same support), go to Step 2.\newline
{\bf Step 2} Let $y_1,\dots, y_r$ be the indeterminates forming the support of all monomials of $F$. Proceed recursively as follows. For all $j=1,\dots, r$, let $\alpha_j$ be the maximum (positive) integer such that $y_j^{\alpha_j}$ divides $F$, and set $F_1=F$. Let $G_1$ be such that $F_1=G_1y_1^{\alpha_1}$. For all indices $j\geq 2$, set $F_j=G_{{j-1}|_{y_{j-1}=0}}$ and let $G_j$ be such that $F_j=G_jy_j^{\beta_j}$, where, for all $j$, $\beta_j$ is the maximum (positive) integer such that $y_j^{\beta_j}$ divides $F_j$. Note that $\beta_j\geq\alpha_j$. Also note that $F_j$ is a nonzero polynomial in the indeterminates $y_h$ with $j\leq h\leq r$, and $G_j$ is a polynomial in the same indeterminates and it contains some monomial not divisible by $y_j$. In particular, $G_r$ is a polynomial in $y_r$ with nonzero constant term ($\ast\ast$).
Now identify each indeterminate $y_j$ involved in Step 2 with the entry of $X$ that replaces $y_j$ in relation ($\ast$), and substitute this entry with $y_j+G_j\Delta$. Call $X'$ the new matrix obtained in this way, and, for all $h=1,\dots, mn-t^2$, set $q'_h=q_h(X')$.\newline
Let $M'$ be the set of $t$-minors of $X'$ other than $\Delta$ (this minor remains unchanged). Then, in view of equation (\ref{2}) we have:
\[
\sqrt{(M')}=\sqrt{(q'_1,\dots, q'_{mn-t^2})}.
\]
We prove that
\[
\sqrt{I_t(X)}=\sqrt{(q'_1,\dots, q'_{mn-t^2})}.
\]
The inclusion $\supset$ is clear, since each $q'_h$ differs from $q_h$ by a multiple of $\Delta$ in $R$. For the inclusion $\subset$ it suffices to prove that whenever, for some ${\bf x}\in K^{mn}$, all polynomials $q'_h$ vanish at ${\bf x}$, then  $\Delta$ also vanishes at ${\bf x}$: in this case the same is true for all polynomials $q_h$, hence, by (\ref{1}),  for all $t$-minors of $X$, and the claim follows by Hilbert's Nullstellensatz. Suppose by contradiction that, under the given assumption, $\Delta$ does not vanish at ${\bf x}$. In the sequel, for the sake of simplicity, we will identify each polynomial with its evaluation at ${\bf x}$; by  a similar convention, we will also call $X'$ the matrix obtained by evaluating all entries of $X'$ at ${\bf x}$. First observe that, since $\Delta\neq 0$, the last $t$ columns of $X'$ are linearly independent. On the other hand, if $i$ is any index such that $1\leq i\leq n-t$, then, for every index $j$ such that $n+t-1\leq j\leq n$, and every sequence of indices $1\leq a_1<\cdots<a_t\leq m$, we also have that, in $X'$,
\[
[a_1,\dots, a_t|i, n-t+1, \dots, \widehat{j},\dots, n]=0,
\]
which implies that each set formed by the $i$-th column and $t-1$ out of the last $t$ columns of $X'$ is linearly dependent. This can only be true if the $i$-th column of $X'$ is zero.  A similar argument can be applied to the rows of $X'$ with indices between 1 and $m-t$. This proves that all entries of $X'$ outside $\Delta$ are zero. Now consider $F$. The entries of $X$ corresponding to the indeterminates $y$ involved in Step 1 remain unchanged; hence they are also entries of $X'$. Moreover, they lie outside $\Delta$, so that they vanish.  It follows that $F$ also vanishes after all iterations of Step 1 are completed. We show that, for all $i=1,\dots, r$, both $y_i$ and $G_i$ are zero; by virtue of ($\ast\ast$), when $i=r$, this provides a contradiction. For the inductive basis note that $y_1+G_1\Delta=0$, so that $y_1\neq 0$ would imply $G_1\neq 0$. But $F=F_1= G_1y_1^{\alpha_1}=0$, thus, in return, we would deduce that $y_1=0$. Since $\Delta\neq 0$, it then follows that $G_1=0$. Now, for some index $h$, $1< h\leq r$, suppose that  $y_{h-1}=G_{h-1}=0$. We then have $0=G_{{h-1}|_{y_{h-1}=0}}=F_h=G_hy_h^{\beta_h}$, whence $G_h=0$ or $y_h=0$. But $y_h+G_h\Delta=0$, which implies $y_h=G_h=0$, as claimed.
\end{proof}

\begin{cor}\label{C.oneZero}
Let $t$ be an integer such that $1\leq t\leq m$ and $t<n$. If some entry of $X$ is zero, then
\[
\ara I_t(X) \leq mn-t^2.
\]
\end{cor}

The next Proposition \ref{P.antidiagonals} provides additional information for the case in which $k=1$: namely, if $\chara K=0$ and $X$ has one zero entry, whereas the remaining entries are pairwise distinct indeterminates, then $\ara I_t(X)=\cd I_t(X)= mn-t^2$.

\section{Improving the upper bound for sparse matrices} \label{S.sparseMatrices}

The result in Corollary \ref{C.oneZero} can be improved for certain classes of matrices having some zero entries. \newline
Given an $(m\times n)$-matrix $A=(a_{ij})$, for all $h=1,\dots, m+n$, the $h$-th {\it antidiagonal} is the sequence $(a_{i,h-i+1})_{i=1,\dots, h}$.

\begin{prop}\label{P.antidiagonals}
Suppose that, for some integer $k$ such that $1\leq k\leq \min\{2t+1,m+n-2t\}$, $k$ entries of $X$, lying on consecutive antidiagonals starting from the left, are zero. Then
\[
\ara I_t(X)\leq mn-t^2-k+1.
\]
\end{prop}

\begin{proof} We perform on $X$ the following recursive procedure. Number the zero entries  according to the following ordering: the zero at $(r,s)$ precedes the zero at $(u,v)$ if either $r+s<u+v$ or $r+s=u+v$ and $r<u$. Thus, in particular, the first zero entry is the one  at $(1,1)$. First obtain a new matrix $X_1$ from $X_0=X$ by replacing the $k$-th zero entry with $p_k=q_{mn-t^2+1}(X)$. For all $h=2,\dots,k$, the $h$-th step consists in obtaining a new matrix $X_h$ from $X_{h-1}$ by replacing the $(k-h+1)$-th zero entry with $p_{k-h+1}=q_{mn-t^2-h+2}(X_{h-1})$.The last step consists in putting $p_1=q_{mn-t^2-k+2}(X_{k-1})$ at (1,1). For all $i=1,\dots, mn-t^2+1$, let $q'_i=q_i(X_k)$. Also set $X'=X_k$.
Fix an index $h$ with $1\leq h\leq k$.  Since $k\leq 2t+1$,  all summands of $p_1,\dots, p_h$ are $t$-minors.   Let $\mu=[a_1,\dots, a_t|b_1,\dots, b_t]$ be a minor appearing as a summand in $p_h$.  According to the above recursive procedure, $p_h$ is inserted in an antidiagonal of index at most $h$. Hence, if $(r,s)$ is its position in the matrix, we have $r+s\leq h+1$. Now, given a $t$-minor greater than $[1,\dots, t|1,\dots, t]$, each of its lower neighbours that is a $t$-minor is obtained by lowering one of its indices by 1. Thus the sum of any subset of indices is lowered at most by $k-h$ when passing from $[m-t+1,\dots, m|n-t+1,\dots, n]$ (the only summand of $p_k$) to $\mu$ (a summand of $p_h$).  In particular, we have that, for all $i,j=1,\dots, t$,
\begin{equation}\label{3}
a_i+b_j\geq m-t+1+n-t+1-k+h=m+n-2t-k+h+2> h+1,
\end{equation}
which means that all entries of $\mu$ lie on antidiagonals with indices greater than $h$. Hence the subsequent steps of the algorithm leave these minors unchanged. Thus, for all $h=1,\dots, k$, $p_h=q'_{mn-t^2-k+1+h}$.
We show that
\[
\sqrt{I_t(X)}=\sqrt{(q'_1,\dots, q'_{mn-t^2-k+1})}.
\]
Note that the polynomials appearing on the right-hand side are those not involved in the above procedure.  We first show the inclusion $\supset$: we inductively prove that, for all $i=1,\dots, k$,  $I_t(X_i)\subset I_t(X_{i-1})$. This is clear for $i=1$. For $i\geq 2$, it suffices to note that every  $t$-minor of $X_i$ differs from the corresponding minor in $X_{i-1}$ by a multiple of $q_{mn-t^2-i+2}(X_{i-1})$. We now prove $\subset$. Once again, we use Hilbert's Nullstellensatz. We assume that at some ${\bf x}\in K^{mn}$ all polynomials $q'_i$, for $1\leq i\leq mn-t^2-k+1$, vanish, and we deduce that the same is true for the remaining (last) $k$ polynomials  $q'_i$, with $mn-t^2-k+2\leq i\leq mn-t^2+1$. Since, for $1\leq i\leq mn-t^2+1$, $q'_i$ differs from $q_i$ by an element of the ideal $(q'_{mn-t^2-k+2},\dots, q'_{mn-t^2+1})$, it will follow that
for $1\leq i\leq mn-t^2+1$,  $q_i$ also vanishes at ${\bf x}$, which, in view of equation (\ref{1}), will imply our claim.  We will proceed inductively on $h\geq 1$, showing that, if  all polynomials $q'_{1}, \dots, q'_{mn-t^2-k+1}, p_1, \dots, p_{h-1}$, vanish at ${\bf x}$, then the same is true for $p_h$. We know that, under the given assumption, all minors of $X_k$ appearing as summands in the polynomials  $q'_{1}, \dots, q'_{mn-t^2-k+1}, p_1\dots, p_{h-1}$ vanish.

Suppose by contradiction that $p_h$ does not vanish at ${\bf x}$. Let $(r,s)$ be the position taken by $p_h$ in the matrix $X'$. Then $r+s\leq h+1$. Let $\mu$ be as above. Thus, by (\ref{3}), we have $a_1+b_1>r+s$, whence $a_1>r$ or $b_1>s$. Suppose that the first case occurs. Then $\mu$ does not involve any entry of the $r$-th row. By assumption all the following minors of $X'$ vanish at ${\bf x}$ (we admit the possibility of repeated column indices, which give rise to zero minors):
\[
[r,a_1,\dots, \widehat{a_i}, \dots, a_t|b_1,\dots, b_t], \quad [r,a_1,\dots, \widehat{a_i}, \dots, a_t|s, b_1,\dots,\widehat{b_j},\dots, b_t],
\]
for $i,j=1,\dots, t$.
The reason is that, in view of (\ref{3}), for each of these minors (without repeated column indices), the sum of all row and column indices is smaller than that of $\mu$, whence each of these minors appears as a summand in one of the polynomials $q'_{1}, \dots, q'_{mn-t^2-k+1}, p_1, \dots, p_{h-1}$. Let $Y$ be the submatrix of $X'$ formed by the rows with indices $r,a_1,\dots, a_t$ and by the columns with indices $s,b_1,\dots, b_t$, evaluated at ${\bf x}$. We have just shown that all sets formed by the $r$-th row of $Y$ and other $t-1$ rows of $Y$ are linearly dependent. But as the $r$-th row of $Y$ is nonzero (because, by assumption, its entry $p_h({\bf x})$ does not vanish), it follows that, in $Y$, the rows with indices $a_1,\dots, a_t$ are linearly dependent. Hence, in particular, $\mu$ vanishes at ${\bf x}$. The same conclusion is drawn if $b_1>s$, by a similar argumentation, in which the roles of rows and columns are interchanged. This proves that $p_h$ vanishes at ${\bf x}$, as desired.
\end{proof}

\begin{ex}
Notice that, if $k$ does not fulfil the upper bound in Proposition \ref{P.antidiagonals}, then it could be that $\ara I_t(X) > mn-t^2-k+1$. In fact, let us consider the matrix
\[
X = \left[\begin{array}{cccc}
0 & x_1 & x_2 & x_3 \\
0 & x_4 & x_5 & x_6 \\
x_7 & x_8 & x_9 & x_{10}
\end{array} \right]
\]
and the ideal of maximal minors $I=I_3(X)$ in the polynomial ring $R=K[x_1,\dots,x_{10}]$, where $\chara K=0$. We show that, in this case, $\ara I=\cd I=3>2$. First of all, $\ara I \leq 12-9=3$ by Corollary \ref{C.oneZero}. Now consider the following Brodmann sequence of local cohomology:
\[
\cdots \longrightarrow H^3_{I+(x_7)}(R) \longrightarrow H^3_I(R) \longrightarrow H^3_I(R_{x_7}) \longrightarrow H^4_{I+(x_7)}(R) \longrightarrow \cdots.
\]
Notice that $I+(x_7)$ is generated by two elements, hence $H^3_{I+(x_7)}(R)=H^4_{I+(x_7)}(R)=0$, so that $H^3_I(R) \simeq H^3_{I}(R_{x_7})$. On the other hand, by virtue of the Independence Theorem \cite[Theorem 4.2.1]{BS13}, $H^3_I(R_{x_7}) \simeq H^3_{I_{x_7}}(R_{x_7})$, where $I_{x_7}=(x_1x_5-x_2x_4,x_1x_6-x_3x_4,x_2x_6-x_3x_5) \subset R_{x_7}$ is the ideal of maximal minors of a generic $(2 \times 3)$-matrix. Therefore $\cd I_{x_7} = 3$ by \cite[Corollary p. 440]{BS90}, so that $H^3_I(R_{x_7}) \neq 0$.  It follows that $H^3_I(R) \neq 0$, whence $\cd I \geq 3$.
\end{ex}

\begin{cor}\label{C.kZeros}
Let $k$ be an integer such that $0\leq k\leq  \min\{2t+1,m+n-2t\}$. Suppose that $X$ has exactly $k$ zero entries, which lie on consecutive antidiagonals starting from the left, whereas the remaining entries are pairwise distinct indeterminates. If $\chara K=0$, then
\[
\ara I_t(X)=\cd I_t(X)=mn-t^2-k+1.
\]
\end{cor}

\begin{proof}
In view of Proposition \ref{P.antidiagonals} it suffices to prove that $\cd I_t(X)\geq mn-t^2-k+1$.  Note that $X$ is obtained from a generic $(m\times n)$-matrix by setting $k$ entries, say $x_1,\dots, x_k$, equal to zero. Let $A_k=R/(x_1,\dots, x_k)$. We show that
\[
H^{mn-t^2-k+1}_{I_t(X)}(A_k)\neq 0,
\]
which will immediately yield the claim. We proceed by induction on $k\geq 0$. For all $h=0,\dots, k$ let $X_h$ be the matrix obtained from the generic matrix by setting $x_1,\dots, x_h$ equal to zero, so that $X_k=X$. The inductive basis is \cite[Corollary, p. 440]{BS90}. So let $k>0$ and suppose that the claim is true for $k-1$. The short exact sequence of $A_{k-1}$-modules
\[
0\to A_{k-1} \mathrel{\mathop{\rightarrow}^{\cdot x_k}}A_{k-1}\to A_k\to 0
\]
induces the following long exact sequence of local cohomology
\[
\cdots\rightarrow H^{mn-t^2-k+1}_{I_t(X_{k-1})}(A_k) \mathrel{\mathop{\rightarrow}^{\delta}} H^{mn-t^2-k+2}_{I_t(X_{k-1})}(A_{k-1}) \mathrel{\mathop{\rightarrow}^{\cdot x_k}} H^{mn-t^2-k+2}_{I_t(X_{k-1})}(A_{k-1}) \rightarrow \cdots
\]
In view of \cite[Lemma 4.1]{LSW16}, $I_t(X_{k-1})_{x_k}$ is the same as $I_{t-1}(Y)$ for some $(m-1)\times(n-1)$ matrix $Y$ with entries in $(A_{k-1})_{x_k}$. Since local cohomology commutes with localization, we thus have that
\[
H^{mn-t^2-k+2}_{I_t(X_{k-1})}(A_{k-1})_{x_k}=H^{mn-t^2-k+2}_{{I_t(X_{k-1})}_{x_k}}((A_{k-1})_{x_k}) = H^{mn-t^2-k+2}_{I_{t-1}(Y)}((A_{k-1})_{x_k}).
\]
But
\[
\ara I_{t-1}(Y)\leq (m-1)(n-1)-(t-1)^2+1<mn-t^2-k+2,
\]
which implies that $H^{mn-t^2-k+2}_{I_{t-1}(Y)}((A_{k-1})_{x_k})=0.$ Thus the module
\[
H^{mn-t^2-k+2}_{I_t(X_{k-1})}(A_{k-1}),
\]
which, by the inductive hypothesis, is nonzero, vanishes when localized at $x_k$.  Therefore the endomorphism $\cdot x_k$ is not injective, so that, in the above long exact sequence, $H^{mn-t^2-k+1}_{I_t(X_{k-1})}(A_k)\neq 0$. But, by virtue of the Independence Theorem, applied to the canonical epimorphism from  $A_{k-1}$ to $A_k$, we have
\[
H^{mn-t^2-k+1}_{I_t(X)}(A_k)\simeq H^{mn-t^2-k+1}_{I_t(X_{k-1})}(A_k),
\]
whence our claim follows.
\end{proof}

\begin{ex}
Consider the polynomial ring $R=K[x_1,\dots, x_7]$, where $\chara K=0$, and the sparse matrix
\[
X = \left[\begin{array}{ccc}
0 & 0 & x_1 \\
x_2 & x_3 & x_4 \\
x_5 & x_6 & x_7
\end{array} \right].
\]
Then $\ara I_2(X)= 9-4-2+1=4$ by Corollary \ref{C.kZeros}. In fact, $I_2(X)=\sqrt{(q_1',q_2',q_3',q_4')}$, where $q_i'=q_i(X')$ and
\[
X' = \left[\begin{array}{ccc}
x_2x_7-x_4x_5+(x_3x_7-x_4x_6)x_7-x_1x_6 & x_3x_7-x_4x_6 & x_1 \\
x_2 & x_3 & x_4 \\
x_5 & x_6 & x_7
\end{array} \right],
\text{ i.e.,}
\]
\begin{align*}
& \begin{aligned}
  q_1' \!=\! [123|123]&=x_4^2x_6^2x_7 - 2x_3x_4x_6x_7^2 + x_3^2x_7^3 + x_1x_4x_6^2 - x_1x_3x_6x_7 - x_1x_3x_5 + x_1x_2x_6
 \end{aligned}\\
& \begin{aligned}
  q_2'\!=\! [12|12]&=  -x_3x_4x_6x_7 + x_3^2x_7^2 - x_3x_4x_5 - x_1x_3x_6 + x_2x_4x_6\\
\end{aligned}\\
& \begin{aligned}
  q_3'\!=\! [12|13]\!+\![13|12] &= -x_4^2x_6x_7 - x_4x_6^2x_7 + x_3x_4x_7^2 + x_3x_6x_7^2 - x_4^2x_5\\
  &-x_1x_4x_6 - x_1x_6^2 + x_2x_4x_7 - x_3x_5x_7 + x_2x_6x_7 - x_1x_2,\\
\end{aligned}\\
& \begin{aligned}
  q_4'\!=\! [12|23]\!+\![13|13]\!+\![23|12]&=  -x_4x_6x_7^2 + x_3x_7^3 - x_4^2x_6 + x_3x_4x_7 - x_4x_5x_7\\
  &-x_1x_6x_7 + x_2x_7^2 - x_1x_3 - x_1x_5 - x_3x_5 + x_2x_6.
\end{aligned}
\end{align*}
\end{ex}

Using arguments similar to those developed in the proof of Proposition \ref{P.antidiagonals} and of Corollary \ref{C.kZeros}, one can show the following result about determinantal ideals of maximal minors.

\begin{prop} \label{P.kZerosCols}
Let $0 \leq k\leq n-m$ and suppose that $k$ entries of $X$ are zero. Then $\ara I_m(X)\leq mn-m^2-k+1$.  If char\,$K=0$ and the remaining entries of $X$ are pairwise distinct indeterminates, then
\[
\ara I_m(X)=\cd I_m(X)=mn-m^2-k+1.
\]
\end{prop}

Proposition \ref{P.kZerosCols} generalizes \cite[Theorem 6.4]{BCMM14}, where the result was proven for $k=m=2$, in the case where the zeros lie on different rows and columns. The technique applied in our proof, however, is completely different, and so are the generators up to radical that we obtain. We present our construction in the following example.

\begin{ex}
Consider the polynomial ring $R=K[x_1, \dots, x_6]$ and the sparse matrix
\[
X = \left[\begin{array}{cccc}
0 & x_1 & x_2 & x_3 \\
x_4 & 0 & x_5 & x_6
\end{array} \right].
\]
If $\chara K=0$, then $\ara I_2(X)= 8-4-2+1=3$ by Proposition \ref{P.kZerosCols}. In fact, $I_2(X)=\sqrt{(q_1',q_2',q_3')}$, where  $q_i'=q_i(X')$, with
\[
X' = \left[\begin{array}{cccc}
x_1x_6-x_3(x_2x_6-x_3x_5) & x_1 & x_2 & x_3 \\
x_4 & x_2x_6-x_3x_5 & x_5 & x_6
\end{array} \right],
\text{ i.e.,}
\]
\begin{align*}
  q_1'&= [12|12]= x_1x_2x_6^2-x_1x_3x_5x_6-x_2^2x_3x_6^2+2x_2x_3^2x_5x_6-x_3^3x_5^2-x_1x_4\\
  q_2'&= [12|13]= x_1x_5x_6-x_2x_3x_5x_6+x_3^2x_5^2-x_2x_4 \\
  q_3'&= [12|14]+[12|23]= x_1x_6^2-x_2x_3x_6^2+x_3^2x_5x_6-x_3x_4 + x_1x_5-x_2^2x_6+x_2x_3x_5.
\end{align*}
\end{ex}

\section{$(2 \times n)$-matrices}\label{S.2nmatrices}

In this section we consider the case of matrices with $2$ rows.

Using Corollary \ref{C.oneZero} we give an affirmative answer to \cite[Question 1]{BCMM14}.

\begin{cor}
Let $n \geq 3$, and let $X$ be a $(2 \times n)$-matrix of linearly dependent linear forms. Then
\[
\ara I_2(X) \leq 2n-4.
\]
\end{cor}
\begin{proof}
According to the Kronecker-Weierstrass theory of matrix pencils (see \cite[Chapter 8]{G59}), the matrix $X$ is equivalent to a concatenation of nilpotent, Jordan and scroll blocks (see \cite[Section 3]{BCMM14} for details). If $X$ contains at least one nilpotent or Jordan block, then one of the entries of $X$ is zero or can be annihilated by elementary row operations, so that $\ara I_2(X) \leq 2n-4$ by Corollary \ref{C.oneZero}. If, on the other hand, $X$ is a concatenation of scroll blocks, then some of them have at least two columns, so that the claim follows by \cite[Theorem 2]{BV10} or by \cite[Theorem 4.2]{BCMM14} (and, for $n>3$, also from Theorem \ref{main}).
\end{proof}

The next result is a consequence of Proposition \ref{P.kZerosCols}. In the special case of a $(2 \times n)$-matrix, the claim is true without any additional restrictions on $k$.
\begin{cor}
Let $n \geq 3$, $k \geq 0$, and let $X$ be a $(2 \times n)$-matrix. Suppose that $k$ entries of $X$ outside the minor $\Delta=[1,2 | n-1,n]$ are zero, then $\ara I_2(X) \leq 2n-3-k$.
If char\,$K=0$ and the remaining entries are pairwise distinct indeterminates, then
\[
\ara I_2(X)=\cd I_2(X)=2n-3-k.
\]
\end{cor}

\begin{proof}
Let $0 \leq h \leq n-2$ be the number of columns of $X$ containing only zero entries. Hence, $I_2(X)=I_2(X')$, where $X'$ is the matrix obtained from $X$ by dropping the zero columns. Thus $X'$ has $n-h \geq 1$ columns. By Proposition \ref{P.kZerosCols} it follows that
\[
\ara I_2(X) = \ara I_2(X') \leq 2(n-h)-3-(k-2h) = 2n-3-k.
\]
The last part of the claim follows from Proposition \ref{P.kZerosCols} applied to $X'$.
\end{proof}

\section{Conclusions}
The proofs of Theorem \ref{main} and Proposition \ref{P.antidiagonals} are based on explicit constructions of the polynomials generating the ideal $I_t(X)$ up to radical. The following example shows a special case in which these methods can be combined in order to produce $mn-t^2-k+1$  polynomials when the matrix $X$ contains $k$ pairwise disjoint sets $S_1$, $S_2,\dots,S_k$ of algebraically dependent entries. The polynomials are obtained by a recursive procedure: it starts at $X=X_0$ and, at the $i$-th step, transforms the matrix $X_{i-1}$ into a new matrix $X_i$ by replacing the entries in $S_i$ according to the algorithm described in the proof of Theorem \ref{main} and taking $F$ to be the polynomial expressing the algebraic dependence in $S_i$. The resulting matrix is $X'=X_k$ and the polynomials $q_1(X'), \dots, q_{mn-t^2-k+1}(X')$ generate the ideal $I_t(X)$ up to radical.
\begin{ex}
Consider the polynomial ring $K[x,y,z,a,b,c,d]$ and the matrix
\[
X = \left[\begin{array}{ccccc}
x^2 &y^2 & z^2 & a & b \\
x^3 & y^3 & z^3 & c & d
\end{array} \right].
\]
The first three columns are sets of algebraically dependent entries. We apply the construction recursively with respect to the third, the second and the first column.
We then obtain
$$\sqrt{I_2(X)}=\sqrt{([12|12],[12|13],[12|14]+[12|23],[12|15]+[12|24])},$$ where the minors on the right-hand side are minors of the matrix
\[
X' = \left[\begin{array}{ccccc}
x^2+w & y^2+u & z^2+(ad-bc) & a &b \\
x^3& y^3 & z^3 & c &d
\end{array} \right],
\]
and we have set
$$w=(y^2+((z^2+(ad-bc)d-z^3b))d-y^3b+(z^2+(ad-bc))c-z^3a, \qquad u=(z^2+(ad-bc))d-z^3b.$$
\end{ex}

We finally propose some open questions.

Let $X$ be an $(m \times n)$-matrix with entries in a polynomial ring over a field $K$. Theorem \ref{main} gives an affirmative answer to \cite[Question 8.1]{LSW16} when the entries of $X$ outside some $(t \times t)$-submatrix of $X$ are algebraic dependent over $K$. But the answer is also known to be true in some cases where the algebraically dependent entries do not fulfil this condition  (see, e.g., \cite[Example 8.3]{LSW16}). The general case is still open.

\begin{problem}
Let $X$ be an ($m \times n$)-matrix whose algebraically dependent entries belong to all $(t \times t)$-submatrices of $X$. Is it true that $\ara I_t(X) \leq mn-t^2$?
\end{problem}

In Corollary \ref{C.kZeros} and Proposition \ref{P.kZerosCols}, we proved that, for some sparse generic matrices with $k$ zero entries, $\ara I_t(X)=\cd I_t(X) =mn-t^2-k+1$ in characteristic zero. For $k=0$, we know from \cite{BS90} that this equality still holds in positive characteristics, but $\cd I_t(X)<\ara I_t(X)$.

\begin{problem}
In the assumptions of Corollary \ref{C.kZeros} and Proposition \ref{P.kZerosCols}, if $\chara K=p>0$, is $\ara I_t(X)=mn-t^2-k+1$?
\end{problem}

Two important classes of matrices are the symmetric and alternating matrices. The cohomological dimension and the arithmetical rank of determinantal and Pfaffian ideals of these matrices have been computed in \cite{B95} and \cite{RW16}.

Let $X$ be an alternating ($n \times n$)-matrix with $k$ symmetric pairs of zero entries outside the main diagonal. If $k$ fulfils the assumptions of Proposition \ref{P.antidiagonals}, using similar arguments, one could compute the arithmetical rank and, in characteristic zero, the cohomological dimension, showing that they are equal and their value is $k$ less than in the generic case.

\begin{problem}
If $X$ is a  symmetric $(n \times n)$-matrix with exactly $k$ zero entries, what are the cohomological dimension and the arithmetical rank of $I_t(X)$? Notice that, for generic symmetric matrices, if $\chara K=2$, $\ara I_t(X)$ depends on the parity of $t$ and for even $t$ it differs from the value of the arithmetical rank in characteristic zero.
\end{problem}

\section*{Acknowledgements}
The first author is indebted to INdAM and NSF for funding. The second author was supported by INdAM.

\end{document}